\documentclass{article}
\usepackage[utf8]{inputenc}
\usepackage[T1]{fontenc}
\usepackage{amsmath,amssymb,amsthm,bbm,float,graphicx,geometry,mathtools,parskip,setspace,subcaption, lmodern}
\usepackage{mathrsfs}
\usepackage{enumerate}
\usepackage{nicefrac}
\usepackage{todonotes}
\usepackage{comment}
\usepackage[colorlinks, linkcolor = blue!80!black, citecolor = blue!80!black, breaklinks, pdfauthor={Benedikt Jahnel, Lukas Luechtrath, Christian Moench}]{hyperref}
\usepackage{orcidlink}
\usepackage{titling}
\usepackage{url}

\usepackage{tikz}
\usetikzlibrary{backgrounds}
\usetikzlibrary{patterns}
\usetikzlibrary{positioning, shapes.geometric}

\usepackage{caption}
\usepackage{graphicx}
\graphicspath{{Images/}}
\allowdisplaybreaks

\newcommand{\N}{\mathbb{N}}

\newtheorem{theorem}{Theorem}[section]

\newtheorem{prop}[theorem]{Proposition}
\newtheorem{conjecture}[theorem]{Conjecture}

\theoremstyle{definition}

\newtheorem{remark}[theorem]{Remark}

\usepackage{nameref}

\makeatletter
\let\orgdescriptionlabel\descriptionlabel
\renewcommand*{\descriptionlabel}[1]{%
  \let\orglabel\label
  \let\label\@gobble
  \phantomsection
  \edef\@currentlabel{#1}%
  \let\label\orglabel
  \orgdescriptionlabel{#1}%
}

\renewcommand{\P}{\mathbb{P}}
\newcommand{\x}{\boldsymbol{x}}

\newcommand{\1}{\mathbbm{1}}
\newcommand{\G}{\mathscr{G}}

\newcommand{\X}{\mathbf{X}}

\newcommand{\R}{\mathbb{R}}
\renewcommand{\d}{\mathrm{d}}

\newcommand{\Z}{\mathbb{Z}}

\renewcommand{\d}{\mathrm{d}}
\newcommand{\bfP}{\mathbf{P}}

\newcommand{\bfE}{\mathbf{E}}
\newcommand{\calG}{\mathcal{G}}
\newcommand{\scrO}{\mathscr{O}}
\newcommand{\scrS}{\mathscr{S}}

\pretitle{\centering\LARGE\scshape}
 \posttitle{\vskip 0.75cm}

 \predate{\vskip 0.5 cm \centering\large}
 \postdate{\par}

\usepackage[style = numeric-comp, sorting=nyt, url = false, abbreviate=false, maxbibnames=9, sortcites=true, doi = true, backend = biber, giveninits = true, isbn=false]{biblatex}
\renewbibmacro{in:}{\ifentrytype{article}{}{\printtext{\bibstring{in}\intitlepunct}}}
\DeclareFieldFormat{journaltitle}{\mkbibemph{#1\isdot}}
\title{Phase transitions for contact processes on one-dimensional networks}

\thanksmarkseries{arabic}

\author{
Benedikt Jahnel \orcidlink{0000-0002-4212-0065}\thanks{Technische Universit\"at Braunschweig, Universit\"atsplatz 2, 38106 Braunschweig, Germany}\thanksgap{0.4ex} \thanks{Weierstrass Institute for Applied Analysis and Stochastics, Mohrenstr.\ 39, 10117 Berlin, Germany} \\ benedikt.jahnel@tu-braunschweig.de
\and
Lukas L\"{u}chtrath \orcidlink{0000-0003-4969-806X}\thanksmark{2}\\lukas.luechtrath@wias-berlin.de \\
\and
Christian M\"{o}nch \orcidlink{0000-0002-6531-6482}\thanks{Johannes Gutenberg-Universität Mainz, Staudingerweg 9, 55128 Mainz, Germany} \\ cmoench@uni-mainz.de
}

\date{\today}

\begin{document}
\maketitle

\begin{spacing}{0.9}
\begin{abstract} 
\noindent We study the survival/extinction phase transition for contact processes with quenched disorder. The disorder is given by a locally finite random graph with vertices indexed by $\Z$ that is assumed to be invariant under index shifts and augments the nearest-neighbour lattice by additional long-range edges. We provide sufficient conditions that imply the existence of a subcritical phase and therefore the non-triviality of the phase transition. Our results apply to instances of scale-free random geometric graphs with any integrable degree distribution. The present work complements previously developed techniques to establish the existence of a subcritical phase on Poisson--Gilbert graphs and Poisson--Delaunay triangulations (M\'enard et al., \emph{Ann.\ Sci.\ Éc.\ Norm.\ Supér.}, 2016), on Galton--Watson trees (Bhamidi et al., \emph{Ann.\ Probab.}, 2021) and on locally tree-like random graphs (Nam et al., \emph{Trans.\ Am.\ Math.\ Soc.}, 2022), all of which require exponential decay of the degree distribution. Two applications of our approach are particularly noteworthy: First, for Gilbert graphs derived from stationary point processes on $\R$ marked with i.i.d.~random radii, our results are sharp. We show that there is a non-trivial phase transition if and only if the graph is locally finite. Second, for independent Bernoulli long-range percolation on $\Z$, with coupling constants $J_{x,y}\asymp |x-y|^{-\delta}$, we verify a conjecture of Can (\emph{Electron.\ Commun.\ Probab.}, 2015) stating the non-triviality of the phase transition whenever $\delta>2$. Although our approach utilises the restrictive topology of the line, we believe that the results are indicative of the behaviour of contact processes on spatial random graphs also in dimensions $d>1$ as long as the degree distribution of the underlying network has at least finite \(d\)-th moment. We support this by proving that no phase transition exists if the \(d\)-the moment is infinite.

\smallskip
\noindent\footnotesize{{\textbf{AMS-MSC 2020}: 60K35 (primary); 05C82, 91D30 (secondary)}

}
\end{abstract}
\end{spacing}

\section{Background and motivation}
We investigate contact processes on spatial random graphs. Let us briefly introduce the contact process and the question of its global survival/extinction, both on fixed and on random graphs. 
\subsection{Survival and extinction of contact processes}
The contact process (sometimes also referred to as the SIS epidemic model) is a Markovian interacting particle system, which was first introduced to the mathematical literature by Harris \cite{harris_contact_1974} and has been studied extensively on hypercubic lattices throughout the 1970s, 1980s and early 1990s. An introduction to this classical theory as well as an overview of its main results is provided in the monographs \cite{liggett_stochastic_1999,liggett_interacting_1985}. In the 2000s renewed interest in the contact process emerged in the context of complex networks, and in particular researchers began to investigate the behaviour of the process on large but finite random graphs as a toy model for the spread of disease or information in a variety of mesoscopic systems, see, for example,~\cite{DurrettDoG2024} for an account of rigorous results in this area.

Together with percolation and oriented percolation, the contact process is one of the simplest random interacting systems that may display a non-trivial phase transition, the existence of which is the subject of the present paper. More precisely, let $G$ be a locally finite connected graph with vertex set $V(G)$ and edge set 
\[
	E(G)\subset V(G)^{[2]}:=\{ \{x,y\}\colon  x\in V(G), y\in V(G), x\neq y\}.
\]
In the contact process on $G$, \emph{infected} vertices in $V(G)$ infect their \emph{susceptible} neighbours at exponential rate $\lambda>0$ and turn back into susceptible vertices at rate $1$. Formally, we view this process as a Markov process with values in the set of subsets of $V(G)$ and generator $L$ given by
\[
Lf(A)= \lambda \sum_{\{u,v\}\in E(G)}\1\{u\in A, v\notin A \}\big(f(A \cup \{v\})-f(A)\big) + \sum_{v\in V(G)}\1\{v\in A \}\big(f(A \setminus \{v\})-f(A)\big),
\]
where $\lambda>0$ is called the \emph{infection rate} and $f$ is any non-negative bounded function on $V(G)$. Starting at time $0$ with an initially infected set $A_0\subset V(G)$, we denote by $\xi^{A_0}_t\subset V(G)$ the set of infected vertices at time $t$ under the dynamics prescribed by $L$. The law of $\xi$ is denoted by $\P_\lambda$. For the purpose of this paper, we always consider \emph{rooted graphs} $(G,o)$, i.e., we distinguish a specific vertex $o\in V(G)$ as the root of $G$. We usually set $A_0=\{o\}$ and omit the initially infected set from the notation in this case. We say that the contact process (with infection rate $\lambda$) \emph{survives on }$G$ if $\P_\lambda (\xi_t\neq \emptyset \text{ for all }t>0)>0$ and define the corresponding \emph{critical infection rate} by
\[
	\lambda_\mathsf{c}(G):=\inf \big\{\lambda>0\colon  \P_\lambda (\xi_t\neq \emptyset \text{ for all }t>0)>0\big\}.
\]
If $\lambda_\mathsf{c}(G)\in(0,\infty)$, then the survival/extinction phase transition for $\xi$ on $G$ is called {\em non-trivial}. 

Since $\emptyset$ is an absorbing state of the dynamics, it follows immediately that $\lambda_\mathsf{c}(G)=\infty$ for any finite graph $G$. Conversely, if $G$ is assumed to be infinite we have \begin{equation}\label{eq:aprioribounds}\lambda_\mathsf{c}(G)\leq \lambda_\mathsf{c}(\mathbb{Z}_{\textup{nn}})<\infty,\end{equation}
where $\mathbb{Z}_{\textup{nn}}$ denotes the nearest-neighbour graph on $\Z$, which we always consider as rooted at $0$. The second inequality in \eqref{eq:aprioribounds} is due to Harris' foundational work~\cite{harris_contact_1974} and the first inequality follows from the fact that $\lambda_\mathsf{c}(\cdot)$ is monotone with respect to the canonical partial order on rooted graphs. The non-triviality of the survival/extinction phase transition thus reduces to the question whether $\lambda_\mathsf{c}(G)>0$, in which case we say that the contact process on $G$ \emph{has a subcritical phase}. A simple comparison with a branching process shows that $\lambda_\mathsf{c}(G)>0$ whenever the vertices of $G$ have a uniformly bounded degree. At the other end of the spectrum, it is not difficult to construct graphs with unbounded degrees on which $\xi$ survives for any $\lambda>0$, for instance by considering a sequence of sufficiently rapidly growing star-graphs whose centre vertices are placed on a copy of the line graph $\mathbb{Z}_{\textup{nn}}$. Moreover, if we interpret the breadth-first exploration of a Bernoulli bond percolation cluster as a contact-type process in which every vertex can only be infected once, then an elementary coupling argument shows that any graph $G$ with bond percolation threshold $p_\mathsf{c}(G)=0$ also satisfies $\lambda_\mathsf{c}(G)=0$.


\subsection{Contact processes on random graphs}

Let us now consider a \emph{random} rooted graph $(\G,o)$ with distribution $\bfP$. We mostly consider cases in which $\G$ is infinite but locally finite $\bfP$-almost surely\footnote{If $\G$ is not locally finite, the contact process can still be defined in the way presented until it infects a vertex of infinite degree. If $\G$ is also connected, this explosion event always has positive probability and we consequently have $\lambda_\mathsf{c}(\G)=0$.}. Furthermore, we assume $\bfP$ to be invariant and ergodic with respect to shifts of the root, e.g., see~\cite{benjamini_ergodic_2012, Khezeli_2018} for discussions of the concept. Note that these assumptions are naturally satisfied in many cases of interest, for instance whenever $\G$ arises as a local limit of a sequence of finite graphs in which the degree of the root is uniformly integrable~\cite{aldous_processes_2018}, or if $\G$ is the infinite cluster of some translation-invariant percolation model on the hypercubic lattice. We focus mainly on the latter case. Note further, that monotonicity of $\lambda_\mathsf{c}(\cdot)$ allows us to assume $(\G,o)$ to be connected, either by incorporating additional edges deterministically or by considering an infinite cluster of the original graph instance under investigation. Ergodicity then ensures that $\lambda_\mathsf{c}\equiv\lambda_\mathsf{c}(\G)$ is non-random.

Let $\mathsf{deg}_{\G}(o)$ denote the number of edges incident to the root of $\G$, or \emph{root degree} for short. We say that $\G$ has \emph{unbounded degrees} if $\bfP(\mathsf{deg}_{\G}(o)\geq k)>0$ for all $k\in\N$. There are surprisingly few examples of random-graph ensembles with unbounded degrees for which the non-triviality of the survival/extinction phase transition for the contact process has been demonstrated. This contrasts with Bernoulli percolation, where the corresponding question is often easier to answer \cite{schulman_long_1983,Lyons1990,penrose_continuum_1991, meester_continuum_1996,deijfen_scale-free_2013,jacob_robustness_2017,gracar_percolation_2021,gracar2023finiteness}. Let us provide some intuition from where this difficulty stems. One way of characterising the percolation phase transition on $\G$ is by looking at sequences of \emph{cut-sets} of edges separating $o$ from infinity. These sets serve as `bottlenecks' for the percolation cluster, and if there are sufficiently many `small' cut-sets, then there must be a subcritical percolation phase. The cut-set characterisation is prominent, for instance, in Lyons' classical works connecting percolation on trees to electrical networks~\cite{Lyons1990,lyons1992random} and in the modern approach of Duminil-Copin and Tassion to the question of sharpness of the percolation phase transition~\cite{duminil2016new,Hugo_sharpnessNN}. It is also closely related to the max-flow/min-cut principle of network optimisation~\cite{ford1962flows}. The infection paths in the contact process, however, are not self-avoiding and in particular may return many times to vertices of high degree, which in turn sustain the infection for a very long time. Intuitively, this means that even if the graph is broken up by cut-sets, the infection gets many attempts at getting through these bottlenecks, if there is an infected vertex of high degree in the vicinity. In general, a far better control on the large-scale geometry of the graph is therefore needed to show that high-degree vertices are sufficiently separated by zones of low connectivity which, in turn, is necessary for the extinction of the contact process.

The random-graph family for which the survival/extinction transition is best understood are Galton--Watson trees. The study of the contact process on Galton--Watson trees was initiated in~\cite{PemStac01}, see also \cite{huang2018contact}. A sufficient and necessary condition for non-triviality of the phase transition is now known due to work of Bhamidi, Nam, Nguyen and Sly \cite{bhamidi_survival_2021}:

\emph{Let $\bfP_\mu$ denote the law of a supercritical Galton--Watson tree $\mathcal{T}$ with offspring distribution $\mu$ conditioned on non-extinction. Then, the contact process on $\bfP_\mu$-almost every $\mathcal{T}$ has a subcritical phase if and only if $\mu$ has an exponential tail.}

\begin{remark}
Since the root vertex of Galton--Watson trees is special, they do not quite fit into our shift-invariant framework. However, if one size-biases the offspring distribution at the root and conditions the tree to be infinite, it is not difficult to see that the resulting infinite \emph{augmented Galton--Watson} tree is distributionally invariant under rerooting; see, for instance, \cite[Chapter~17]{PTRNbook}. Since these modifications have no effect on the existence of exponential tails for the offspring distribution, the result of Bhamidi et al.\ has an equivalent formulation in terms of augmented Galton--Watson trees conditioned on survival.
\end{remark}

The above result for Galton--Watson trees also has consequences for locally tree-like graphs~\cite{bhamidi_survival_2021,nam_critical_2022}, which include a variety of popular network models, such as configuration graphs and inhomogeneous random graphs. Outside the locally tree-like setting, the subcritical phase of the contact process on random graphs is rather poorly understood. The only approach available so far is based on the use of a technically challenging auxiliary process called \emph{cumulative merging percolation}, which was introduced by Ménard and Singh~\cite{menard_percolation_2016}. This approach essentially formalises the above heuristics about quantifying the distributions of zones of good and bad connectivity in the graph. Ménard and Singh applied their approach to the contact process on both the Gilbert graph and the Delaunay triangulation obtained from a homogeneous Poisson point process in $\R^d$. More precisely, they showed the following:

\emph{Let $\bfP$ denote the distribution of either one of the following random graphs,
\begin{itemize}
    \item the origin's cluster in the Gilbert graph with supercritical radius $r\in(0,\infty)$, derived from the Palm version of a unit intensity Poisson point process on $\R^d$, conditioned on being infinite; or
    \item the Delaunay triangulation derived from the Palm version of a unit intensity Poisson point process on $\R^d$ rooted at the origin,
\end{itemize}
then, the corresponding contact process has a subcritical phase $\bfP$-almost surely.}

Shortly after, Can~\cite{can_cp_lrp} applied the same technique to long-range percolation clusters on $\Z$. To state his result, which is the starting point of our own investigation, we first provide an explicit definition of the underlying random graph $\G$. We set $V(\G)=\Z$ and choose $0$ as the root. Each edge $\{x,y\}$ is now included into $E(\G)$ independently of all other edges with probability
\[
\bfP\big(\{x,y\}\in E(\G)\big)=\varphi(|x-y|),\quad \{x,y\}\in\Z^{[2]}.
\]
We call $\varphi$ the \emph{connection function} of the long-range percolation graph. Of particular interest is the case when 
\begin{equation}\label{eq:connect}
\varphi(k)=k^{-\delta} \text{ for some } \delta>1 \text{ and all }k\in\N.    
\end{equation}
 In this setting, we define
\[
\delta^\ast:= \inf\{\delta>1\colon \bfP\text{-almost every } \G \text{ admits a subcritical phase for the contact process}\}.
\]
Can conjectured that $\delta^\ast\leq 2$ and obtained the following bound \cite{can_cp_lrp}: 

\emph{For the contact process on long-range percolation graphs induced by connection functions of the type~\eqref{eq:connect}, the exponent $\delta^\ast$ satisfies}
\begin{equation}\label{eq:cansresult}
\delta^\ast\leq 102.
\end{equation}

\subsection{Our contribution}
We consider the problem of existence of a phase transition for contact processes on stationary random graphs in \emph{one dimension}. Although infinite Galton--Watson trees are arguably the simplest infinite random graphs \emph{without geometry}, our setting lies at the opposite end of the spectrum as we consider \emph{complete confinement by geometry}. In contrast to previous investigations of contact processes on random graphs, which use heavy technical machinery, we obtain in Theorem~\ref{thm:exsubphase} a rather general statement by technically effortless means. Aside from improving upon Can's result~\eqref{eq:cansresult} for long-range percolation graphs, the main contribution of this work, as we see it, is twofold.

First, we provide a `proof of concept': As is illustrated by our Theorems~\ref{thm:DBP} and~\ref{thm:WDRCM}, in the geometric setting, there are simple and ergodic random-graph models producing sparse graphs with \emph{arbitrarily heavy-tailed degree distributions} on which the contact process nevertheless undergoes a non-trivial phase transition. This is in stark contrast to the non-spatial Galton--Watson setting. Moreover, all previous results of this type for spatial models only apply to graphs with exponentially decaying degree tails and heavily rely on independence or asymptotic independence of graph neighbourhoods -- besides the existence of zones of low edge density, we only require invariance, ergodicity and finite expected root degree, see Theorem~\ref{thm:exsubphase}.

The second, and more technical, innovation of the present work is the coupling with a random walk in random environment that we develop to prove Theorem~\ref{thm:exsubphase}. It replaces the `cumulative merging' approach previously used on random geometric graphs and we believe it to be of independent interest.

Admittedly, our technique heavily utilises the restrictive geometry of the line and cannot be readily generalised to higher dimensions. We believe that, in more than one dimension, a result of comparable scope as Theorem~\ref{thm:exsubphase} is currently out of reach. However, we are convinced that the presence of a phase transition on random geometric graphs in which connection probabilities decay sufficiently fast in the vertex distance is \emph{not} an artefact of the one-dimensional setting, and we discuss this matter in greater detail at the end of Section~\ref{sec_WDRCM} where we also give a necessary condition on the degree distribution for the phase transition to exists.

\paragraph{Paper overview.} The remainder of the paper is structured as follows: Our results are presented and discussed in Section~\ref{sec_results}. There, we first state our main result, Theorem~\ref{thm:exsubphase}, and then its implications for specific models, such as long-range percolation and Gilbert graphs derived from Boolean models with additional edges as well as weight-dependent random connection models. Section~\ref{sec_proofs} is dedicated to proofs.

\section{Main results and discussion}\label{sec_results}
We proceed with formulating our main results and explain how they relate to previous work and the general spatial case. We begin with the key result of this work, which provides a general sufficient criterium for the existence of a phase transition in the contact process on stationary random graphs embedded in $\Z$. Subsequently, we focus on the consequences of our approach for a number of models that have received attention in the literature.

Consider a random graph $\G=(\Z,E)$ rooted at $0\in\Z$ with law $\bfP$ that satisfies the following assumptions:
\begin{description}
    \item[(A.1)\label{A:stat}] $\G$ is \emph{stationary and ergodic}, i.e., $\bfP$ is invariant and ergodic with respect to additive shifts along $\Z$.
    \item[(A.2)\label{A:sparse}] $\G$ is \emph{sparse}, i.e., we have $\Delta_\G:=\bfE[\mathsf{deg}_{\G}(0)]<\infty$.
\end{description}
Note that, without loss of generality, we may 
henceforth assume that $\G$ is connected: If $\G$ is stationary, then so is the graph obtained from $\G$ by adding all nearest-neighbour edges. Furthermore, ergodicity and sparsity imply \emph{local finiteness} as they ensure that $\bfP$-almost surely no vertex has infinitely many neighbours. 
\begin{remark}
The reason that we equate sparsity with finite expected root degree is that the total number of edges $E_n$ of the subgraph induced by $\{0,\dots,n-1\}$, under Assumptions~\ref{A:stat} and~\ref{A:sparse}, satisfies
\begin{equation}\label{eq:edgedensity}
\lim_{n\to\infty}\frac{E_n}{n}= \frac{1}{2}\Delta_\G\quad \mathbf{P}\text{-almost surely},
\end{equation}
by the pointwise ergodic theorem. Consequently, the number of edges is locally of the same order as the number of vertices precisely if the mean degree of the root is finite.
\end{remark}
We now set up a decomposition of $\G$ inspired by the cut-set characterisation of percolation alluded to in the introduction. We distinguish between nearest-neighbour pairs $\ell_z=\{z-1,z\}, z\in\Z$ which we call \emph{links} and the edges of $\G$. As indicated by the notation, we usually identify links with their \emph{larger} endpoint. Not every link necessarily needs to correspond to an edge of $\G$, but as explained earlier, there is no harm in assuming that $\mathbb{Z}_{\textup{nn}}$ is a subgraph of $\G$. For a given link $\ell_z, z\in\Z$, we define the \emph{number of edges above} $\ell_z$ as
\[
\mathsf{e}(z)=\big|\big\{\{x,y\}\in E\colon x\in(-\infty,z-1], y\in[z,\infty) \big\}\big|\in \N\cup\{0,\infty\}.
\]
We say $z\in \Z$ is a $K$-\emph{cut point} (for $\G$), if there are at most $K\ge 1$ edges of $\G$ above $\ell_z$, i.e., $\mathsf{e}(z)\leq K$. Since we assume $\G$ to be connected, we have that $\mathsf{e}(z)\leq 1$ if and only if $\mathsf{e}(z)=1$. We set $\mathsf{e}=\mathsf{e}(0)$ and focus on the case in which the random variable $\mathsf{e}$ under $\bfP$ is almost-surely finite, in which case $0$ has a positive probability of being a $K$-cut point for some $K<\infty$. It is worth noting that interesting models fulfilling~\ref{A:stat} and~\ref{A:sparse} exist for which \(\mathsf{e}\) is infinite almost-surely and consequently no cut points are present in the graph. However, the existence of cut points is crucial for our method. Let us now state our main theorem for models with cut points.

\begin{theorem}[Sufficient criterium for the existence of the survival/extinction phase transition] \label{thm:exsubphase}
Let $\G$ satisfy the Assumptions~\ref{A:stat} and~\ref{A:sparse} and let $\mathsf{e}$ be $\bfP$-almost-surely finite.
Then, 
\[
	\lambda_{\mathsf{c}}(\G)\equiv \lambda_{\mathsf{c}}\in(0,\infty)\qquad\text{$\bfP$-almost-surely. 
}
\]
\end{theorem}
\begin{remark}\label{rem:schul}
\begin{enumerate}[(i)]
    \item Note that Assumptions~\ref{A:stat} and~\ref{A:sparse} combined imply that finiteness of $\mathsf{e}$ is a zero-one event. Indeed, using sparsity, \(\bfP(\mathsf{e}=\infty)=\bfP(\mathsf{e}(z)=\infty\text{ for all }z\in \Z)\) and since the event \(\{\mathsf{e}(z)=\infty\) for all \(z\in \Z\}\) is translation invariant, under ergodicity, \(\bfP(\mathsf{e}=\infty)\in\{0,1\}\). Furthermore, observe that the finiteness of \(\mathsf{e}\) assumption generalises Schulman's classical condition for the absence of percolation in one-dimensional long-range percolation models~\cite{schulman_long_1983}. Consequently, the `typical' random graph models that fall under the scope of Theorem~\ref{thm:exsubphase} do not have infinite clusters and adding edges is necessary to obtain an infinite environment for the contact process.
    \item The condition $\mathbf{P}(\mathsf{e}<\infty)=1$ implies that long-range edges in $\G$ are so few, that $\G$ remains essentially $1$-dimensional. In particular, $\mathbf{P}(\mathsf{e}<\infty)=1$ implies $p_\mathsf{c}(\G)=1$, and hence the percolation phase transition on $\G$ is trivial. The converse is not true: take $\G$ to be scale-invariant long-range percolation at the critical point \cite{AizenmanNewman86}, then clearly $p_\mathsf{c}(\G)=1$ but $\mathbf{P}(\mathsf{e}=\infty)=1$, since $\mathsf{e}$ is a sum of independent Bernoulli random variables with $\bfE[\mathsf e]=\infty$. 
\end{enumerate}
 
\end{remark}


\subsection{Extinction phases for specific graph models}\label{sec:examples}
We now apply our main result to a variety of graph models from the literature, namely, long-range percolation, augmented Gilbert graphs as well as weight-dependent random connection models. 
\subsubsection{Long-range percolation} We begin by improving Can's result for long-range percolation models, thereby verifying his conjecture from \cite{can_cp_lrp}.
\begin{theorem}[Extinction in long-range percolation] \label{thm:LRP}
Let $\G$ denote the graph obtained by performing independent Bernoulli long-range percolation on $\Z$ with connection function $\varphi\colon \N\to [0,1]$ satisfying
\[
\sum_{k=1}^\infty k \varphi(k)<\infty\quad \text{ and }\quad\{k\colon  \varphi(k)=1\}\neq \emptyset.
\]
Then, $\lambda_{\mathsf{c}}(\G)\equiv \lambda_{\mathsf{c}}>0$ $\bfP$-almost surely. In particular, we have that
\(\delta^\ast\leq 2\).
\end{theorem}
In fact, we believe that this result is still not sharp,
the conclusion of Theorem~\ref{thm:LRP} should remain valid whenever
\[
\sum_{k=1}^\infty \varphi(k)<\infty.
\]
In particular this would imply $\delta^\ast=1.$
The intuition behind this prediction comes from the fact that sparse i.i.d.\ Bernoulli percolation graphs have a degree distribution with all exponential moments. We know from the result of Bhamidi et al.~\cite{bhamidi_survival_2021} that this is a sufficient condition for the non-triviality of the survival-extinction phase transition in Galton--Watson trees. As the latter can be viewed as spread-out limits of long-range percolation clusters, see e.g.~\cite{Bollo_spread_07}, it is hard to conceive how imposing the additional restriction of a low-dimensional ambient geometry could help the contact process survive. 

Shortly after we prepared the first version of the present article, an alternative proof of Theorem~\ref{thm:LRP} using a Peierls-type argument was given by Gomes et al.\ \cite{gomes2025extinctioncontactprocessonedimensional}.

\subsubsection{Augmented Gilbert graphs} The next model we discuss is a variant of the Gilbert graph associated with Boolean percolation models~\cite{hall_continuum_1985,gouere_subcritical_2008}. Let $X=\{\dots,X_{-1},X_0=0,X_1,\dots\}$ denote a stationary and ergodic point process on $\R$ under its Palm-distribution, i.e., with a point $X_0=0$ at the origin. We assume that $X$ is indexed according to the natural order on $\R$. To each $X_j\in X$, we assign a random radius $R_j$ drawn from a given distribution $\rho$ on $(0,\infty)$. Note that the radii are not assumed to be independent, it suffices that the joint process of points and radii be stationary and ergodic under additive shifts in $\R$. The graph $\G$ is then obtained by setting $V(\G)=\Z$ and
\[
\{i,j\}\in E(\G) \quad\Longleftrightarrow\quad |X_i-X_j|\leq R_i+R_j \text{ or }|i-j|=1,
\]
which is equivalent to taking the Gilbert graph associated with the Boolean model induced by the radii, adding edges between consecutive vertices and identifying the vertices with their index, while keeping all edges. We call this model the \emph{augmented Gilbert graph}.
\begin{theorem}[Extinction in augmented Gilbert graphs] \label{thm:DBP}
Let $(X_i,R_i)_{i\in\Z}\subset \R\times[0,\infty)$ denote a stationary and ergodic sequence such that $X=(X_i)_{i\in\Z}$ forms a simple point process and such that the law $\rho$ of $R_0$ satisfies
\begin{equation}
	\begin{aligned}\label{eq:vertexmarks}
		\int_0^{\infty}r\rho(\d r)<\infty.
	\end{aligned}
\end{equation}
Let furthermore $\G$ denote the corresponding augmented Gilbert graph. Then, $\lambda_{\mathsf{c}}(\G)\equiv \lambda_{\mathsf{c}}>0$ $\bfP$-almost surely, and thus the contact process on $\G$ has a subcritical phase.
\end{theorem}
Let us note that the requirement that $R_0$ be integrable cannot be weakened since otherwise the underlying Graph model will in general not be locally finite. In that sense, we believe that the condition in Theorem~\ref{thm:DBP} is indeed sharp and for classical Gilbert graphs derived from i.i.d.\ radii the existence of an extinction phase for the contact process on the augmented graph coincides with the existence of a connectivity phase transition of the unaugmented graph with respect to the intensity of the underlying point process.
\begin{theorem}[Survival and extinction in augmented Gilbert graphs with i.i.d.\ radii] \label{thm:DBP2}
Let $X=(X_i)_{i\in\Z}\subset \R\times[0,\infty)$ denote a stationary and ergodic simple point process endowed with i.i.d.\ marks $(R_i)_{i\in\Z}$ drawn from a distribution $\rho$ on $[0,\infty)$ 
independently of $X$. Then, the contact process on the induced augmented Gilbert graph $\G$ has a subcritical phase $\bfP$-almost surely if and only if
\[
	\int_0^{\infty}r\rho(\d r)<\infty.
\]
\end{theorem}

\subsubsection{Augmented weight-dependent random connection models}\label{sec_WDRCM}
Finally, let us extend the application of our main theorem to \emph{weight-dependent random connection models} (WDRCMs), introduced in~\cite{gracar_recurrence_2022}, which generalises both previous examples. We use the framework of~\cite{jacob2024annulus} and assume \(X=(X_i)_{i\in\Z}\) to be either a stationary and ergodic simple point process under its Palm distribution, or the lattice \(\Z\). Next, let \(U=(U_i)_{i\in\Z}\) be a sequence of independent random variables distributed uniformly on \((0,1)\), and define \(\X=((X_i,U_i))_{i\in\Z}\). Note that \(\X\) is an independently marked stationary and ergodic simple point process under its Palm distribution. We build an auxiliary graph \(\calG\) on the points of \(\X\) from which we will derive the graph \(\G\). To this end, let \(\varphi\colon (0,1)\times(0,1)\times(0,\infty)\to [0,1]\) be a \emph{connection function} that is symmetric in the first two arguments and non-increasing in all three arguments. Now, given \(\X\), we connect any pair of vertices \(\X_i,\X_j\) in \(\calG\) independently with probability
\begin{equation} \label{eq:WDRCM}
	\bfP\big(\{\X_i,\X_j\}\in E(\calG) \, \big| \, (X,U)=(x,u)\big) = \varphi(u_i,u_j,|x_i-x_j|).	
\end{equation}

Note that the non-increasing property of the marks arguments imply that vertices with a mark closer to zero have a higher probability of being connected. Hence, the mark models the \emph{inverse weight} of a vertex, giving the model its name. Similarly, short edges are more probable than long edges. The augmented graph \(\G\) is now obtained from \(\calG\) by setting
\[
	V(\G)=\Z \quad \text{ and } \quad E(\G) = \big\{\{i,i+1\}\colon i\in\Z\big\}\cup \big\{\{i,j\}\colon \{\X_i,\X_j\}\in E(\calG), i\neq j\in\Z\big\}.  
\]
It is clear that \(\G\) satisfies the Conditions~\ref{A:stat} and~\ref{A:sparse} if and only if \(\calG\) does. Furthermore, the existence of cut points, i.e., the finiteness of $\mathsf{e}$, 
in \(\G\) is determined by the edge set of \(\calG\), and thus by the properties of \(\varphi\) in particular.

\begin{theorem}[Extinction in augmented weight-dependent random connection models] \label{thm:WDRCM}
	Let \(\G\) be the augmented graph of a WDRCM, constructed on the points of the Palm version of a stationary and ergodic simple point process \(\X\) with connection function \(\varphi\). Suppose that there exists \(\mu>0\) such that
	\begin{equation} \label{eq:WDRCMcut}
	  	\sum_{n\in\N} 2^{2n} \int\limits_{2^{-n-\mu n}}^1 \int\limits_{2^{-n-\mu n}}^1 \varphi(u,v,2^n) \, \d u \, \d v <\infty,
	\end{equation}
 	then \(\lambda_c(\G)=\lambda_c>0\) \(\bfP\)-almost surely and thus, the contact process on \(\G\) has a subcritical phase. 
\end{theorem}
Let us mention one example, for which Theorem~\ref{thm:WDRCM} applies, namely
\[
    \varphi(u_i,u_j,|x_i-x_j|) = \1\{|x_i-x_j|\leq u_i^{-\gamma} u_j^{-\gamma}\}.
\]
For this choice,~\eqref{eq:WDRCMcut} is fulfilled if and only if \(\gamma<1/2\). Before augmenting this graph, it serves as the weak local limit of hyperbolic random graphs after appropriate change of variables~\cite{komjathy_explosion_2020}. The contact process on hyperbolic random graphs, corresponding to \(\gamma>1/2\) has been extensively studied in~\cite{linker_contact_2021}, where it is shown that, in the limiting graph, there is no extinction phase. While for \(\gamma>1/2\) there also is no subcritical percolation phase, for $\gamma=1/2$, such a percolation phase transition exists~\cite{bode_largest_2015,gracar_percolation_2021}. It would be interesting to further investigate the contact process at this regime boundary, in particular because the absence of a survival/extinction phase transition there would imply the same for the regime boundary in the one-dimensional version of the model studied in~\cite{gracar2024contact}.

\subsection{Potential extensions to higher dimensions.}\label{sec:extensionsdisc}
Property~\eqref{eq:WDRCMcut}, which simultaneously guarantees sparseness~\ref{A:sparse} and the existence of cut points embodies a general theme of the study in the WDRCMs in general, namely that the presence of many long edges drastically changes the model's properties, see also Remark~\ref{rem:schul}. As elaborated in~\cite{gracar2023finiteness}, the \(n\)-th summand in Property~\eqref{eq:WDRCMcut} essentially determines the expected number of long edges directly connecting the sets \(\{\X_{-2^{n}},\dots,\X_{-2^{n-1}}\}\) and \(\{\X_{2^{n-1}},\dots,\X_{2^n}\}\). As the event of \(\X_0\) being a (\(1\)-)cut point is equivalent to the absence of any such edge, summability ensures that \(\X_0\) is a cut point with positive probability, hence implying finiteness of \(\mathsf e\). 
It is important to note that truncating the lower integral bound in~\eqref{eq:WDRCMcut}  is crucial to exclude atypically powerful vertices that would dominate the expected number of long edges without actually being present. Specifically, the expected number of vertices in the two considered sets, that have mark smaller than \(2^{-n-n\mu}\), is \(2^{-2\mu n}\), which again is summable. 

Although the relationship between long edges and cut points is inherently one-dimensional, the concept of quantifying long edges in a graph via the tail of an integral, as in~\eqref{eq:WDRCMcut}, can be generalised to higher dimensions. For the existence of a subcritical percolation phase, this strategy has proven successful in the recent work~\cite{jacob2024annulus}. The geometric arguments in that paper are based on crossing probabilities for annuli of the form \(\{x\in\R^d\colon n<|x|\leq 2n\}\). This is a natural extension of the cut-point concept: Consider an edge \emph{long} if it connects a vertex within the inner ball \(\{|x|\leq n\}\) to some vertex located outside of the annulus. In~\cite{jacob2024annulus}, it is established that, if long edges asymptotically vanish, then annulus crossings characterise the existence of a percolation phase transition. More precisely, by optimising an integral formula similar to~\eqref{eq:WDRCMcut}, the authors exactly quantify the occurrence of long edges and establish two regimes: Either long edges appear frequently, allowing large annuli to be crossed by a single edge, or long edges are rare, in which case large annuli can only be traversed via relatively long paths. The latter scenario can always be suppressed by significantly decreasing the underlying intensity parameter, ultimately preventing large annuli from being traversed. Therefore, a subcritical percolation phase exists. Similar geometric restrictions apply to the spread of the contact process, which makes it very plausible that the absence of long edges is sufficient for the existence of a subcritical phase also in WDRCMs in higher dimension.

This is further corroborated by another important consequence of the scarcity of long edges, namely the comparability of graph distances with Euclidean distances. In~\cite{luechtrath2024chemical}, it is shown that, if long edges are rare, then the graph distance between any two vertices at a large distance is bounded from below by a multiple of their Euclidean distance. The result holds for \emph{any} sparse translation-invariant graph in which long edges are rare. This behaviour contrasts sharply with the non-spatial case, where graph distances typically grow no faster than logarithmically in the system size, see~\cite{Bollobook,ChungLuDist03,RemcoDist05}.

On an intuitive level, this phenomenon may imply the existence of an extinction phase for graphs with power-law degree distributions in the presence of sufficiently strong geometric constraints. More precisely, although high-degree vertices, on which the contact process survives locally for a long time, exist, these hubs should be too far apart to apply the ``chain of stars'' strategy, which guarantees supercriticality on heavy-tailed trees in~\cite{bhamidi_survival_2021}. 
However, the higher the dimension, the less geometric restrictions are felt by the graph. Let us explain this effect, by way of example, for the Boolean model in dimension \(d\geq 2\)  as a special instance of a WDRCM. For technical simplicity, we construct the model directly on the lattice \(\Z^d\). More precisely, the vertex set is given by the lattice sites \(z\in\Z^d\), where each site is independently assigned a uniform mark \(u_z\). The connection function~\eqref{eq:WDRCM} in this case reads
\[
	\varphi(u_z,u_v,|z-v|) =\1\{|z-v|\le u_z^{-\gamma/d}+u_v^{-\gamma/d}\}, \quad \gamma\in(0,1).
\]
This is essentially the augmented Boolean model above with Pareto\((d/\gamma)\)-distributed radii but directly constructed on \(\Z^d\). In particular,
\begin{enumerate}[(i)]
	\item all nearest-neighbour edges are present.
	\item For two disjoint subsets \(B_1,B_2\subset \Z^d\), the induced subgraphs \(\G\cap B_1\) and \(\G\cap B_2\) are independent.
	\item The tail of the degree distribution is determined by the volume of associated balls of large radii and thus follows the power-law with exponent \(\tau=1+1/\gamma\)~\cite{gracar_transience_2020,gracar_age-dependent_2019}, i.e., 
    \[
        \bfP(\deg(o,U_o)=k)\sim k^{-\tau+o(1)}, \quad \text{ as } k\to\infty.
    \]
\end{enumerate} 
Consider now the subgraph induced by a large box of volume \(L\). By Observation~(iii), the largest degree in that box is roughly of order \(L^{1/(\tau-1)}=L^\gamma\). By~(i) and~\cite{luechtrath2024chemical}, the graph distance of the highest-degree node to the boundary of the box is of the order of the side length \(L^{1/d}\). Note that the degree does not depend on the dimension while the graph distance does. Therefore, taking two instances of the model, using the same \(\gamma\) but two different dimensions \(d_1<d_2\), both observed in a respective box of volume \(L\), the higher dimension should be beneficial for the contact process. This is due to the fact that the largest-degree vertex of the box in both cases survives for roughly the same time, which is determined by \(L^\gamma\), while distances are shorter in the second case as \(L^{1/d_1}>L^{1/d_2}\). Making this comparison between degree distribution and distances precise leads to the following theorem.

\begin{theorem}\label{thm:higherDim}
	Let \(\G\) be the Boolean model on \(\Z^d\) thus constructed with \(d\geq 2\) and \(\gamma\in(0,1)\). If \(\gamma>1/d\) (or equivalently \(\tau<d+1\)), then \(\lambda_\mathsf{c}(\G)=0\) \(\bfP\)-almost surely and the contact process on \(\G\) has no subcritical phase.
\end{theorem}

It is not very difficult to deduce that the result applies to all instances of WDRCMs that have the Properties~(i)--(iii) by adapting our proof in Section~\ref{sec:higherDim}. Furthermore, the restriction to the vertex set \(\Z^d\) ensures connectivity of the whole graph allowing to exhibit the principal statement without extra technical challenges. However, the proof suggests that the result would not change if we worked on the infinite cluster of the standard Boolean model constructed on a Poisson point process instead.

In essence, Theorem~\ref{thm:higherDim} identifies a relation between the degree distribution and the dimension that makes the existence of a subcritical phase for the contact process impossible. Note that the relation \(\gamma>1/d\) is never satisfied in \(d=1\) in the sparse regime \(\gamma<1\), which is consistent with our Theorem~\ref{thm:DBP2}. On the other hand, the relation is always satisfied for the other boundary case \(d=\infty\) for any \(\gamma>0\). If we think of the non-geometric tree case as the \(d=\infty\) boundary case, this is again consistent with the results of~\cite{chatterjee_contact_2009,bhamidi_survival_2021}. Theorem~\ref{thm:higherDim} then shows how the geometric influence on the contact process on a power-law graph becomes less restrictive as the dimension grows. 

Let us briefly discuss the implications of the above on the existence of a subcritical phase for the contact process on a WDRCM with power-law exponent \(\tau\) in any dimension \(d\geq 2\). By Theorem~\ref{thm:higherDim}, \(\tau>d+1\) is necessary for the existence of a subcritical phase. Moreover, in our heuristic, we compared the degree distribution with graph distances that were of the same order as the Euclidean distances. If the graph distances were of much smaller order, say polynomially in the logarithm of the Euclidean distance, then again the geometric restrictions should not be felt by the stars and a subcritical phase should fail to exist, just as in the non-spatial case, where typical distances are of logarithmic order or shorter. Therefore, graph distances that are comparable to Euclidean distances (which is always the case in the sparse Boolean model for which Theorem~\ref{thm:higherDim} is formulated) should also be necessary for the existence of a subcritical phase on a power-law graph. As already mentioned above, this is closely related to the absence of annulus crossings via single long edges. Using the condition of~\cite{jacob2024annulus, luechtrath2024chemical} describing the absence of these kind of long edges and combining it with our findings, ultimately leads to the following conjecture.

\begin{conjecture}
	Let \(\G\) be a supercritical WDRCM with power-law degree distribution
	\[
		\bfP(\deg(o)=k)\sim k^{-\tau+o(1)}, \quad \text{ as }k\to\infty,
	\]
	for \(\tau>d+1\), and few long edges in the sense that
	\[
		\bfP(\exists x\sim y\colon |x|<n, |x-y|>2n) \leq n^{-\zeta+o(1)}, \quad \text{ as }n\to\infty,
	\]
	for \(\zeta>0\). Then, we have \(\lambda_\mathsf{c}(\G)=\lambda_\mathsf{c}>0\) \(\bfP\)-almost surely and the contact process on \(\G\) has a subcritical phase.
\end{conjecture}

\section{Proofs}\label{sec_proofs}

\subsection{Proof of Theorem~\ref{thm:exsubphase}}\label{sec:1stProof}
We begin by proving the theorem under the stronger assumption of existence of $K$-cut points with $K=1$ and explain how to obtain the general case at the end of the section. 
For this, we say that $z\in\Z$ is a {\em cut point} if $\mathsf{e}(z)=1$ and the edge above $\ell_z$ is $\{z-1,z\}$. If cut points exist, i.e., \(\bfP(\tau<\infty)=1\), where
\[
\tau =\min\{z\geq 0\colon z \text{ is a cut point}\},
\]
they must have positive density due to ergodicity. In particular, the distribution
\[
\overline{\bfP}(\cdot)=\bfP(\,\cdot\;|0\text{ is a cut point})
\]
is well-defined, and the contact process with infection rate $\lambda>0$ dies out almost surely on $\bfP$-almost every realisation of $\G$ if and only if it dies out almost surely on $\overline{\bfP}$-almost-every $\G$. We denote by $z_k,k\geq 1,$ the $k$-th cut point to the right of $z_0=0$. By ergodicity, the sequence of cut points is well-defined and
\[
\lim_{n\to\infty}\frac{\big|\big\{z\in\{0,\dots,n\}\colon z \text{ is a cut point} \big\}\big|}{n} =: p\qquad\text{${\bfP}$-almost surely,}
\]
 where \(p\) denotes the probability (w.r.t.\ \(\bfP\)) that \(0\) is a cut point. Define, for $k\geq 1$, $C_k:=\{z_{k-1},\dots,z_k-1\}$ to be the {\em $k$-th block} in the partition of $\Z$ derived from the cut points (and $0$) and let $\mathcal{C}_k$ be the subgraph induced by $C_k$ in $\G$.
\begin{prop}\label{prop:cutting}
Let $\G$ satisfy the Assumptions~\ref{A:stat} and~\ref{A:sparse}, and let $p=\bfP(0 \text{ is a cut point})>0$. Then, the sequence $(\mathcal{C}_k)_{k\geq 0}$ is stationary under $\overline{\bfP}$ and we have
\begin{equation}\label{eq:finiteblockexp}
\overline{\bfE}|C_1|=1/p<\infty \quad\text{ as well as }\quad \overline{\bfE}|C_1|^2=(1+2\bfE[\tau])/p.
\end{equation}
\end{prop}
\begin{proof}
Stationarity of $(\G,0)$ under $\bfP$ yields (`cycle'-)stationarity of $(\G,0)$ under $\overline{\bfP}$ if the origin is shifted to consecutive cut points, see~\cite{Thorisson95}. In particular this implies the stationarity statement for the subgraphs $\mathcal{C}_k, k\geq 1$. The fact that $\overline{\bfE}|C_1|=1/p$ follows from Kac's Lemma \cite{Kac47} applied to the sequence $(\1\{z \text{ is a cut point}\})_{z\in\Z}$, see \cite{Kasteleyn86} for a short proof in the case of binary sequences that suffices for our purpose. Further, the expression for $\overline{\bfE}|C_1|^2$ is found in~\cite[Equation~(5)]{Kac47} or can be obtained directly by size-biasing the distribution of $|C_1|$ to move from $\overline{\bfP}$ to $\bfP$.
\end{proof}
\begin{remark}
The conclusion of Proposition~\ref{prop:cutting} is in fact still essentially valid without the assumption of ergodicity as long as there exists a cut point with positive probability. The bound on $\overline{\bfE}|C_1|$ then becomes
\[
	\overline{\bfE}|C_1|=\frac{1-\bfP(\text{there is no cut point})}{p}<\infty,
\]
see \cite[Equation~(2)]{Kasteleyn86}.
\end{remark}
We further require that, under $\overline{\mathbf{P}}$, the expected number of edges per block is finite.
\begin{prop}\label{prop:finedge}
Let $\G$ satisfy the Assumptions~\ref{A:stat} and~\ref{A:sparse}, and let $p=\bfP(0 \text{ is a cut point})>0$, then
\[
\overline{\bfE} |E(\mathcal{C}_1)|<\infty.
\]
\end{prop}
\begin{proof}
By~\eqref{eq:edgedensity}, the global edge density of $\G$ is $\Delta_\G/2$ $\bfP$-almost surely, and thus the same is true $\overline{\bfP}$-almost surely. In particular, under $\overline{\bfP}$ we have for any $\varepsilon>0$
\begin{equation}\label{eq:edgebound1}
\frac{\Delta_\G}{2}\geq\limsup_{n\to\infty}\frac{\sum_{i=1}^n |E(\mathcal{C}_i)|}{\sum_{i=1}^n |C_i|}\geq \lim_{n\to\infty}\frac{\sum_{i=1}^n |E(\mathcal{C}_i)|}{(1+\varepsilon)n\overline{\bfE}|C_1|},
\end{equation}
since 
\[
\lim_{n\to\infty} \frac{1}{n}\sum_{i=1}^n |C_i|=\overline{\bfE}|C_1|<\infty
\]
by Proposition~\ref{prop:cutting} and the pointwise ergodic theorem. Taking expectations in~\eqref{eq:edgebound1} and letting $\varepsilon\to 0$ yields
\[
\overline{\bfE} |E(\mathcal{C}_1)|\leq \frac{\Delta_\G}{2}\overline{\bfE} |C_1|=\frac{\Delta_\G}{2 p}<\infty,
\]
as desired.
\end{proof}
Finally, our proof uses a coupling between the contact process and a \emph{random walk in random environment}. To make this coupling effective, we require a classical result of Ledrappier. The setting is as follows: Consider a stationary and ergodic sequence $(\omega_z)_{z\in\N}$ with distribution $\mathbf{Q}$, called the \emph{environment}. Conditionally on the environment, the random walk $S$ is now defined as the $\mathbf{Q}$-almost surely well-defined Markov chain $(S_t)_{t\in\N}$ with $S_0=0$ and generator
\[
\1{\{x>0\}}\big[\omega_x\big(f(x-1)-f(x)\big)+(1-\omega_x)\big(f(x+1)-f(x)\big)\big]+\1{\{x=0\}}\big(f(1)-f(0)\big),
\]
for $x\in \{0,1,2,\dots\}$.

\begin{prop}[{Ledrappier \cite{Ledrappier82}, cf.\ \cite{Zeit04}}]\label{prop:ledrappier}
${S}$ is recurrent if and only if 
\[
	\int\log \Big(\frac{1-{\omega}_1}{{\omega}_1}\Big)\,\mathbf{Q}(\d\omega_1) \geq 0,
\]
provided this expectation is well-defined.
\end{prop}
\begin{remark}
The standard version of this result, given in \cite{Ledrappier82,Zeit04}, is in fact formulated for the two-sided variant of the random walk in random environment. Proposition~\ref{prop:ledrappier} is a simple corollary of the standard version, obtained by reflecting the walk at $0$.
\end{remark}

As a key step, we now give the proof for the special case that cut points exist and generalise this below to the general case, where only the existence of \(K\)-cut points, for some \(K\geq 1\), is required.   

\begin{prop}\label{prop:exsubphase_K=1}
Let $\G$ satisfy the Assumptions~\ref{A:stat} and~\ref{A:sparse}, and assume $\bfP(\tau<\infty)=1$.
Then, $\bfP$-almost-surely 
\[
	\lambda_{\mathsf{c}}(\G)\equiv \lambda_{\mathsf{c}}\in(0,\infty).
\]
\end{prop}

\begin{proof}
We may work under the distribution $\overline{\bfP}$, since the almost-sure extinction of the contact process is a graph property that does not depend on where the initially infected vertex lies. Let us further modify the usual contact process $\xi$ on $\G$ in such a way that the initially infected set is $A_0=\{-1,0\}$ and that $A_0$ \emph{remains permanently infected}. It follows immediately that the resulting contact processes on the non-negative and negative half-line, respectively, evolve independently of each other and have the same distribution. Thus it suffices to consider the contact process ${\xi}^\dagger$ on the subgraph of $\G$ induced by $\{0,1,2,\dots\}$.

Note that the process $\xi^{\{-1,0\}}$ dies out $\P_\lambda$-almost surely for $\overline{\bfP}$-almost-every realisation of $\G$, if 
\begin{equation}\label{eq:daggerdiesIO}
 \P_\lambda\Big(\bigcap_{n\geq 1}\bigcup_{t>n}\big\{\xi^\dagger_t=\{0\}\big\} \Big)=1.   
\end{equation}

To show that \eqref{eq:daggerdiesIO} holds, consider another auxiliary process $\eta=(\eta_t)_{t\geq 0}$. Its dynamics are those of ${\xi}^\dagger$, but following the convention that, at the update times of the \emph{right-most infected vertex}, all vertices to the left of it become infected instantaneously. Further define 
\[
    X_t:=\max\eta_t,
\]
and let $0=J_0,J_1,J_2\dots$ denote the jump times of $X_t$. Then $\eta$ evolves according to the Markovian dynamics as $\xi^\dagger$ on the time intervals $(J_{k-1},J_k], k\geq 1$, but jumps to state $\{0,\dots,X_{J_k}\}$ instantaneously at the times $J_k$, $k\geq 1$. By attractiveness of the contact process, we have that $\eta$ dominates $\xi^\dagger$, and thus~\eqref{eq:daggerdiesIO} is implied by
\begin{equation}\label{eq:etadiesIO}
 \P_\lambda\Big(\bigcap_{n\geq 1}\bigcup_{t>n}\big\{\eta_t=\{0\}\big\} \Big)=1.   
\end{equation}

Our next step is to introduce a coarse-grained version of the right-most particle process $X=(X_t)_{t\ge 0}$. Define
\[
Y_t:= k \text{ if }X_t\in C_{k},
\]
and note that the $k$-th block $\mathcal{C}_k$ is followed by the $k$-th cut point $z_k$. Let $J'_n,$ $n\geq 0$, denote the jump times of $(Y_t)_{t\ge 0}$ and define the \emph{discrete block process}
\[
Z_n:=Y_{J'_n},\quad n\in \{0,1,\dots\}.
\]
This final reduction allows us to deduce~\eqref{eq:etadiesIO} from the recurrence of $Z=(Z_n)_{n\geq 0}$. To obtain recurrence, we now use the crucial observation that, by Assumption~\ref{A:sparse} and Proposition~\ref{prop:cutting}, $Z$ is a random walk in a stationary ergodic random environment on $\N\cup\{0\}$, which is induced by the realisation of $\G$.

Let us introduce, for $k>0$, the transition probabilities
\[
\P_\lambda(Z_{k+1}=z+1|Z_{k}=z)=:\omega_k(\lambda).
\]
Observe that $\omega_k(\lambda)$ is a random variable under $\overline{\bfP}$ that depends only on the structure of the $k$-th block $\mathcal{C}_k$. Clearly, if $X_t\in C_{k}$, the particle configuration inside 
the $k$-th block that maximises the probability that $X$ jumps to the $(k+1)$-st block in the next update is to have every single vertex in $C_{k}$ infected. Let us assume that we are in this configuration and that, without loss of generality, $\lambda<1$. Then, the time that $X$ needs to traverse the edge above $\ell_{z_{k}}$ is bounded from below by an $\operatorname{Exp}(1)$ random variable by virtue of $z_k$ being a cut point. It follows by attractiveness of the contact process that, given $\G$, the probability that the whole $k$-th block recovers before $X$ jumps to the $k+1$-st block is at least
\begin{equation}\label{eq:lowerboundomegak}
\textup{e}^{-|C_{k}|-2\lambda |E(\mathcal{C}_k)|}\leq 1-\omega_k(\lambda), \quad k\geq 1.
\end{equation}
%
Now, note that $\omega_1(\lambda)\downarrow 0$, almost surely, as $\lambda\downarrow 0$, since the first block is a finite graph. Hence, 
\[
\begin{aligned}
    \overline{\bfE}\log\big((1-\omega(\lambda))/\omega(\lambda)\big) 
    & 
        = \overline{\bfE}\log(1-\omega(\lambda))-\overline{\bf{E}}\log\omega(\lambda)
    \geq -\overline{\bfE}(|C_1|+2\lambda|E_1|) +\overline{\bfE}\log(1/\omega(\lambda)),
    \\ &
        \geq -\overline{\bfE}(|C_1|+2 |E_1|) +\overline{\bfE}\log(1/\omega(\lambda)),
\end{aligned}
\]
by~\eqref{eq:lowerboundomegak} and the fact that we restricted ourselves to \(\lambda<1\). By Propositions~\ref{prop:cutting} and~\ref{prop:finedge}, the first expectation is finite, while the second one tends to \(+\infty\), as \(\lambda\downarrow 0\), by monotone convergence; Proposition~\ref{prop:ledrappier} yields the desired recurrence. 
%
\end{proof}

\begin{remark}
Consider an i.i.d.\ sequence $K_1,\dots,K_2$ of copies of a random variable $K$ supported on $\N$. Identify every vertex $n$ of the nearest-neighbour lattice $\N$ with the root of a rooted clique of size $K_n$. Then, the above proof shows that the contact process has a non-trivial phase transition on the induced graph $\G$ if $\bfE K^2<\infty$. On the other hand, $\bfE K^2=\infty$ implies that
\[
\bfE\log(1-\overline{\omega}_1(\lambda))=-\infty
\]
for any $\lambda>0$ and the fact that each block consists of at least one vertex provides the trivial uniform bound
\[
\bfE\log(\overline{\omega}_1(\lambda))\geq \log\big(1-\textup{e}^{-1}\big)>-\infty.
\]
Hence, by Ledrappier's original theorem, e.g., in the form of~\cite[Theorem~2.1.2]{Zeit04}, the corresponding $\overline{Z}$ process is transient for any $\lambda$. We believe that the contact process always survives in this case. Studying this toy model in higher dimensions might be instructive before attempting to generalise the results of this work to multi-dimensional models with a more complex cluster structure. 
\end{remark}

We conclude this section by lifting the previous domination argument from cut points to $K$-cut points.
\begin{proof}[Proof of Theorem~\ref{thm:exsubphase}]
Assume for the moment, that we can decompose $\G$ stationarily into consecutive disjoint blocks of vertices $C_k$ corresponding to induced subgraphs $\mathcal{C}_k$ such that there are at most $K$ edges joining each block to its successor. Then, only a minor modification in the proof of the $K=1$ case is necessary, namely the traversal time between blocks is now uniformly lower bounded by an $\operatorname{Exp}(K)$ random variable. This has obviously no qualitative effect on the argument. Hence it suffices to show that such a decomposition exists, which does not trivially follow from the existence of $K$-cut points if $K\neq 1$.

To construct such a decomposition, we say that $z\in \Z$ is a \emph{$(K,L)$-cut point}, if there are precisely $K$ edges above $\ell_z$ and none of them is longer than $L-1$. It follows from translation invariance and local finiteness of $\G$ that, if $K$-cut points exists, so do $(K,L)$-cut points for sufficiently large $L$. Conditionally on $z_0=0$ being a $(K,L)$-cut point, we recursively declare $z_k$ to be the $(L+1)$-st $(K,L)$-cut point after $z_{k-1}, k\geq 1$. Then, the induced block structure has the desired properties and this concludes the argument.
\end{proof}

\subsection{Analysis of specific random graph models}\label{sec:proof-examples}
It remains to verify the assumptions of Theorem~\ref{thm:exsubphase} for our example models to prove Theorems~\ref{thm:LRP}--\ref{thm:WDRCM}.
\begin{proof}[Proof of Theorem~\ref{thm:LRP}]
That $\G$ is stationary and ergodic follows from translation invariance of the edge probabilities and independence of edges. Sparsity is equivalent to $\sum_{k}\varphi(k)<\infty$. Consider the link $\ell_0=\{-1,0\}$, then, the number of edges above is given by 
\[
	\mathsf{e}=\sum_{x=-\infty}^{-1}\sum_{y=0}^\infty \chi(x,y),
\]
where each $\chi(x,y)$ is $\operatorname{Bernoulli}(\varphi(y-x))$-distributed and the $\chi(x,y)$ are all independent. Since $\sum_k k\varphi(k)<\infty$ by assumption, it follows that
\[
\mathbf{P}(\mathsf{e}=1)>0,
\]
and thus cut points have positive density.
\end{proof}
\begin{proof}[Proof of Theorems~\ref{thm:DBP} and~\ref{thm:DBP2}]
Consider the underlying Boolean model under its stationary distribution. Assume for the moment that radii can be arbitrarily small, i.e., $\inf\textrm{supp}(\rho)=0$. It is a classical result of continuum percolation theory, see e.g.~\cite{meester_continuum_1996}, that the probability that $0$ is not covered by any ball of the Boolean model is positive if and only if $\int r \rho(\d r)<\infty$. Hence, under the Palm distribution and after augmentation, $0\in V(\G)=\Z$ has a positive probability of being a cut point. Now assume that $\inf\textrm{supp}(\rho)\leq K$ for some integer $K$. Then, the same reasoning applies but with the property of `being covered' is replaced by the property of `being covered by at most $2K$ balls. This leads to a positive density of $2K$-cut points, which suffices to deduce Theorem~\ref{thm:DBP} from Theorem~\ref{thm:exsubphase}. Finally, suppose the radii of the balls of the Boolean model are independent -- this is the classical setting for continuum percolation. In this case
\[
	\int r \rho(\d r)=\infty
\]
implies that $0$ has infinite degree in $\G$, see~\cite{meester_continuum_1996}, and hence the contact process trivially survives at any positive intensity, which establishes Theorem~\ref{thm:DBP2}.
\end{proof}
\begin{remark}
The i.i.d.\ assumption in Theorem~\ref{thm:DBP2} can be relaxed, since the proof of almost everywhere infinite degree given in~\cite{meester_continuum_1996} also works under less restrictive assumptions. However, we are not aware of a simple method to relax them to the minimal assumptions of Theorem~\ref{thm:DBP} and this is not in the scope of the present work.
\end{remark}
\begin{proof}[Proof of Theorem~\ref{thm:WDRCM}]
	We again prove Theorem~\ref{thm:WDRCM} by verifying the assumptions of Theorem~\ref{thm:exsubphase}. Stationarity and ergodicity subject to~\ref{A:stat} follow from the stationarity of the underlying vertex set \(\X\), and translation invariance and conditional independence of the edges. The existence of cut points is a direct consequence of Property~\eqref{eq:WDRCMcut} and~\cite[Section~4]{gracar2023finiteness}, which has already been outlined above. It thus remains to prove sparsity~\ref{A:sparse}. Using independence of the marks, Campbell's formula~\cite{last_lectures_2018} and stationarity of the underlying point process, we obtain
	\[
		\begin{aligned}
			\Delta_\G 
			&
				= 2 + \sum_{n=0}^\infty \int_{0}^1 \bfE\Big[ \sum_{\X_j\in \X} \1{\{\{\X_j,\X_0\}\in E(\calG)\}}\1{\{2^n<|X_j|\leq 2^{n+1}\}} \, \big| \, U_0=u \Big] \d u
			\\ &
				= 2+ 2\sum_{n=0}^\infty \int_0^1 \int_0^1 \int_{2^n}^{2^{n+1}} \varphi(u,v,x) \, \d x \, \d u  \, \d v.
		\end{aligned}
	\]   
	Using that \(\varphi\) is non-increasing in the distance argument and~\eqref{eq:WDRCMcut}, we thus infer
	\[
		\begin{aligned}
			\Delta_\G 
			&
				\leq 2 + 4\sum_{n=0}^\infty \Big(2^{-n-2\mu n} + 2^n\int_{2^{-n-n\mu}}^1 \int_{2^{-n-n\mu}}^1 \varphi(u,v,2^n) \, \d u \, \d v\Big) <\infty,	
		\end{aligned}
	\] 
	as desired. 
\end{proof}

\subsection{Absence of subcritical phase in higher dimensions in the presence of heavy-tailed degrees.} \label{sec:higherDim}
In this final section, we prove Theorem~\ref{thm:higherDim}, the absence of a subcritical phase when the degrees are sufficiently heavy-tailed.
\begin{proof}[Proof of Theorem~\ref{thm:higherDim}.]
	We employ a block-renormalisation argument. Fix any infection rate \(\lambda\in(0,1)\) for the contact process on \(\G\), which will remain unchanged during the course of the proof. For \(L\in\{2^d,3^d,\dots\}\) let \(B_L(v)\), \(v\in\Z^d\), be the box \(L^{1/d}v+(-L^{1/d}/2, L^{1/d}/2]^d\) of volume \(L\), centred at \(L^{1/d} v\). Clearly, \((B_L(v)\colon v\in\Z^d)\) provides a disjoint partition of \(\Z^d\). We plan to compare the contact process on \(\G\) to an \emph{oriented bond-percolation process} on a large supercritical Bernoulli site-percolation cluster of \(\Z^d\). Let us denote by \(\Z^d_p\) an instance of i.i.d.\ Bernoulli$(p)$-site percolation on the nearest-neighbour lattice \(\Z^d\) and by \(\scrO:=\scrO_{p,q}\) an oriented bond-percolation process on \(\Z^d_p\times\N_0\), where each oriented nearest-neighbour edge of the form \((v,n)\to (v\pm e_i,n+1)\) is open with probability \(q\) conditionally on the involved vertices being open in $\Z^d_p$. Here, \(v\in\Z^d_p\), \(n\in\N_0\), and \(e_i\) denotes the \(i\)-th unit vector, where $i\in\{1,\dots,d\}$. Let us think of the case of independent oriented edges first and note that \(\scrO\) contains an oriented path to infinity almost surely if \(p\) and \(q\) are large enough. This can easily be seen by first choosing \(p>p_c^{\text{site}}(\Z^d)\). Then, \(\Z^d_p\) contains an infinite self-avoiding path \((v_0,v_1,v_2,\dots)\). Identifying \(v_i\) with \(i\in\N_0\), \(\scrO\) contains an oriented percolation process on \(\N_0\times\N_0\), and it is well-known that the latter contains an oriented path to infinity for sufficiently large \(q\). Let us next consider the generalisation to \(K\)-dependent edges, that is, whether an edge is open or not only depends on the \(K\)-neighbourhood of the edge. Then, an application of the well-known result of Ligget, Schonmann, and Stacey~\cite{ligget1997_Domination} yields the same result, perhaps requiring an increased \(q\).
	The block-renormalisation argument now works as follows:
    \begin{itemize}
        \item A block \(B_L(v)\) is {good} if its induced subgraph \(\G\cap B_L(v)\) contains a {star} and a vertex of particular high degree. Each good box corresponds to an open site in \(\Z^d_p\). \item We then consider potential infection paths from the star in a good box to the star of a neighbouring good box in a predefined time interval  and identify these paths with the oriented edges.  
	\end{itemize}
    Let us already mention that boxes are good independently of each other, because the radii of the balls associated with each box are independent. Similarly, the existence of the considered infection paths are at most two-dependent, as we will show below.
    
	Before we give the precise construction, let us collect some known properties of star graphs. Let \(\scrS_k\) be the star graph with root \(o\) and \(k\) leaves. Recall that \(\lambda<1\) and consider \(\xi_t^\ell(\scrS_k)\), the contact process on \(\scrS_k\) starting with the root \(o\) and \(\ell\) leaves being infected. Set \(K= \lceil k \lambda/(1+2\lambda)\rceil\) and fix any \(\varepsilon_1>0\), then \cite[Lemma~2.6]{huang_contact_2020-1} ensures the existence of two constants \(c_1=c_1(\lambda,\varepsilon)>0\) and \(c_2=c_2(\lambda,\varepsilon)>0\) such that
	\begin{equation}\label{eq:star1}
		\P_\lambda \Big(\inf_{t\leq T}|\xi_t^K(\scrS_k)|\geq  \varepsilon_1 K \Big)\geq 1- \mathrm{e}^{-c_1 k}, \quad \text{ where } T=\mathrm{e}^{c_2 k}.
	\end{equation}
	In words, if a certain proportion of the leaves and the centre is infected in the beginning, then the infection survives on \(\scrS_k\) exponentially long while keeping a positive proportion of leaves infected with a probability that is exponentially large in the star's size. Moreover, it is easy to deduce that~\eqref{eq:star1} remains valid even if the centre is not infected at time \(0\). Indeed, the number of leaves that recover until the first infection of \(o\) is geometrically distributed with parameter \(1/\lambda\), see also~\cite{huang_contact_2020-1}. In particular, the probability that the number of recovered leaves exceeds \(k/1000\) decays exponentially in \(k\). Due to the strong Markov property, we can restart the process at the first time of infection of the centre and apply~\eqref{eq:star1}, potentially slightly adapting the constants involved. Finally,~\cite[Lemma~2.6]{huang_contact_2020-1} yields that
	\begin{equation}\label{eq:star2}
		\P_\lambda\Big(\inf_{k^{2/3}\leq t\leq T}|\xi_t^0(\scrS_k)|\geq \varepsilon_1 K\Big)\geq 1-c_3 k^{-1/3},
	\end{equation}
	for sufficiently large \(k\), where \(c_3=c_3(\lambda)>0\). That is, if only the centre is infected at time \(0\), then the contact process still has a large probability to infect a positive proportion of its leaves until time \(k^{2/3}\) that persists exponentially long if the star is sufficiently large. 
	
	Let us finally move to the core of our proof, the coupling of the contact process on \(\G\) and the oriented percolation process \(\scrO\). As outlined above, we identify each box \(B_L(v)\) of \(\G\) with the site \(v\in\Z^d\). Choose some \(\varepsilon\in(0,\gamma-1/d)\), which is possible by assumption. We call the box \(B_L(v)\) \emph{good} if there is some \(\x\in B_L(v)\) with \(\deg_{B_L(v)}(\x)\geq L^{\gamma-\varepsilon}\), where \(\deg_{B_L(v)}(\x)\) denotes the degree of \(\x\) restricted to the subgraph of $\G$ induced by \(B_L(v)\). We call the vertex of largest degree in a good box the \emph{star} of \(B_L(v)\) and denote it by \(o_v\). Clearly, whether a box is good or not is independent of all other boxes. Moreover, there exist dimension-dependent constants \(c,C>0\) such that \(B_L(v)\) contains a star if \(B_{cL}(v)\) contains a vertex of radius no smaller than \(CL^{(\gamma-\varepsilon)/d}\). Hence, by independence of the radii,
	\[
		\begin{aligned}
			\bfP(B_L(v) \text{ is good}) 
			&
				\geq \bfP\big(\exists\, \x\in B_{cL}(v)\colon U_x^{-\gamma}\geq C L^{\gamma-\varepsilon}\big)
				=  1- \bfP\big(\forall \, \x\in B_{cL}(v)\colon U_x^{-\gamma}< C L^{\gamma-\varepsilon}\big)
			\\ &
				\geq 1-\mathrm{e}^{-L^{\varepsilon'}}=:p_L,
		\end{aligned}
	\]    
	for some \(\varepsilon'>0\). In particular, \(B_L(v)\) is good with a probability arbitrarily close to $1$, if \(L\) is chosen large enough. We now declare a site \(v\in\scrO\) open if \(B_L(v)\) is good, to produce a supercritical configuration \(\Z^d_{p_L}\) with positive probability. Next, given a realisation of good boxes and corresponding open sites, respectively, we now define the oriented bond-percolation process. This is most straightforward in the graphical representation of the contact process, see~\cite{liggett_interacting_1985}. Denote by \(\scrS_v\) the star graph induced by \(o_v\) and its neighbours in the good box \(B_L(v)\) and set \(k_L=L^{\gamma-\varepsilon}\). To utilise the above results about star graphs, we define a sequence of discrete time steps \((t_n)_{n\in\N}\) with \(t_0=0\) and \(t_{n+1}-t_n=\lfloor \mathrm{e}^{c_2 k_L}\rfloor\) for all \(n\in\N\). For the sake of readability, we omit the Gauß brackets from here on onwards. We now consider a nearest-neighbour pair \((v,v+e_i)\) and consider the oriented edges \(((v,n)\to(v+e_i,n+1))_{n\in\N}\) (treating the pair \((v,v-e_i)\) analogously). Let \(\hat{\xi}_t\) be the contact process restricted to \(\G\cap(B_L(v)\cup B_L(v+e_i))\), where we suppress the dependence of \(v\) and \(e_i\) for convenience, and let \((\hat{\xi}_{n,t})_{n\in\N}\) be a sequence of independent copies of \(\hat{\xi}_t\). By the Markov property, for any \(t<t_1\) and \(n\in\N\), we have that \(\hat{\xi}_{t+t_n}\) and \(\hat{\xi}_{n,t}^{\hat{\xi}_{t_n}}\) follow the same law. We now declare the oriented edge \((v,n)\to(v+e_i,n+1)\) open with a probability that solely depends on \(\hat{\xi}_{n}\) guaranteeing independence of the edges among the time steps. Let us denote by \(\hat{\xi}_{n,t}^\ell(\scrS_v)\) the contact process restricted to the star graph of \(B_L(v)\) with initially \(\ell\in\N\) vertices of it infected. If \(\ell=0\) we assume the root \(o_v\) to be infected. We declare the edge \((v,n)\to(v+e_i,n+1)\) open, if
\begin{enumerate}[(i)]
	\item 
		we have \(\inf_{t<t_1}|\hat{\xi}_{n,t}^{K_L}(\scrS_{v})|\geq \varepsilon_2 K_L\), with \(K_L=\varepsilon_1 k_L \lambda/(1+2\lambda)\) for some fixed \(\varepsilon_2<\varepsilon_1\). That is, the infection persists with positive density on \(\scrS_v\) during the time interval \([t_n,t_{n+1})\) if initially enough of its vertices have been infected.
	\item
		There exists \(\tau<t_{1}-k_L^{2/3}\) such that \(o_{v+e_i}\in\hat{\xi}_{n,\tau}^{K_L}\). That is, if a large enough proportion of \(\scrS_v\) is infected at time \(t_n\), the infection reaches the centre \(o_{v+e_i}\) of the star in the neighbouring box before time \(t_{n+1}-k_L^{2/3}\). 	
	\item 
		On the event \(\tau<t_{1}-k_L^{2/3}\) of Item~(ii), we have that at least \(\varepsilon_1 K_L\) vertices of the star \(\scrS_{v+e_i}\) are infected at time \(t_1\), if we start the infection at time \(\tau\) with only the centre being infected. Note that this has the same probability as the event \(|\hat{\xi}_{n,t_1-\tau}(\scrS_{v+e_i})|\geq \varepsilon_1 K_L\). That is, if there exists an infection path defined in~(ii), then the star \(\scrS_{v+e_i}\) builds mass until time \(t_{n+1}\). 		 
\end{enumerate}
Combining,~\eqref{eq:star1} and~\eqref{eq:star2} with the strong Markov property, both Item~(i) and~(iii) hold each with probability at least \(1-\delta/3\) for any \(\delta\) if \(L\) is chosen large enough. It thus remains to verify Item~(ii). To this end, recall that all nearest-neighbour edges are present in \(\G\). Hence, there exists a path from \(o_v\) to \(o_{v+e_i}\) no longer than \(2^{1+1/d}L^{1/d}\) that is completely contained in \(B_L(v)\cup B_L(v+e_i)\). Consider the event, that the path directly transmits the infection without using any edge twice. That is, each vertex \(v_j\) of the path infects its subsequent vertex \(v_{j+1}\) before recovering itself. By the memory-less property of the Exponential distribution, this happens with probability \(\lambda/(1+\lambda)\). Therefore, the probability, that the path can directly be traversed without any recovery event happening is bounded by \((\lambda/(1+\lambda))^{2^{1+1/d} L^{1/d}}\). Now, conditionally on the Event~(i), a positive proportion of \(\scrS_v\) remains infected during the whole considered time interval. By an elementary Poisson concentration argument, there exists some \(c_4=c_4(\lambda)>0\) such that \(o_v\) recovers and gets reinfected at least \(c_4 k_L\) times before time \(t_1-k_L^{2/3}\). If we start an independent new trial to pass the infection directly through the path after each reinfection of \(o_v\), then we may infer from attractiveness of the contact process~\cite{liggett_interacting_1985} that the probability that Event~(ii) happens is bounded below by the probability that a Binomial random variable with \(\mathrm e^{c_4 k_L} = \mathrm{e}^{c_4 L^{\gamma-\varepsilon}}\) trials and success probability \((\lambda/(1+\lambda))^{2^{1+1/d} L^{1/d}}\) is strictly positive. The expectation of that Bernoulli random variable satisfies
\begin{equation}\label{eq:infectionPath}
	\exp\Big(c_4 L^{\gamma-\varepsilon} - 2^{1+1/d}\log((1+\lambda)/\lambda)L^{1/d}\Big) \longrightarrow \infty,
\end{equation}
as \(L\to\infty\), by the assumption \(\varepsilon<\gamma-1/d\). Combining these observations, the Events~(i),~(ii), and~(iii) all occur simultaneously with probability no smaller than \(1-\delta\) for \(L\) large enough. Furthermore, note that the infection event only considers the contact process in a specific time interval with specified initial conditions, specific infection paths and the star graphs at a path's end and ignores all other edges of \(\G\) and all other infected sites. Therefore, the oriented edges are spatially at most two dependent and, as outlined above, independent across time. Thus, choosing \(\delta\) small enough and \(L\) large enough, the corresponding oriented percolation process \(\scrO_{p_L,1-\delta}\) contains an infinite oriented path as explained above. In particular, the origin \(0\in\scrO\) has a positive probability of starting an infinite oriented path. By construction, such a path yields an infinite infection path in the graphical construction of the contact process, if we start with the whole box \(B_L(0)\) being infected. As \(\lambda\) was chosen arbitrarily and \(\lambda_\mathsf{c}\) does not depend on the \emph{finite} initial condition, this implies \(\lambda_\mathsf{c}=0\). Let us mention for the last step that, again by attractiveness, the original contact process on \(\G\) dominates the union of all the contact processes that were used in the previous step.    
\end{proof}

\noindent\textbf{Acknowledgement}. We thank Peter Gracar and Marco Seiler for insightful discussions about the research presented here. We further would like to extend our gratitude to the organisers of the workshop \emph{Long-range Phenomena in Percolation} at the University of Cologne in September 2024 for providing a very stimulating environment for these discussions to take place. Finally, we thank Johannes B\"aumler for pointing out a flaw in our predictions for the presence of a phase transition in higher dimension in an earlier version, which motivated us to prove Theorem~\ref{thm:higherDim}. 

\noindent\textbf{Funding acknowledgement.} CM's research is funded by Deutsche Forschungsgemeinschaft (DFG, German Research Foundation) – SPP 2265 443916008. BJ and LL acknowledge the financial support of the Leibniz Association within the Leibniz Junior Research Group on \emph{Probabilistic Methods for Dynamic Communication Networks} as part of the Leibniz Competition as well as  the Berlin Cluster of Excellence {\em MATH+} through the project {\em EF45-3} on {\em Data Transmission in Dynamical Random Networks}.

\section*{References}
\renewcommand*{\bibfont}{\footnotesize}
\printbibliography[heading = none]

\end{document}